\documentclass[12pt]{amsart}
\usepackage[margin=1.5in]{geometry}

\usepackage[utf8]{inputenc}
\usepackage{amsmath,amsfonts,amssymb,mathrsfs,amstext,amscd,latexsym,amsthm, mathtools}
\usepackage{hyperref}
\usepackage{enumerate}
\usepackage{comment}
\usepackage{xcolor}
\usepackage{pgfplots}

\usepackage{graphicx}
\usepackage{subcaption}

\def\R{\mathbb{R}}
\def\Z{\mathbb Z}

\def\E{\mathbb{E}}
\DeclareMathOperator{\spa}{span}
\DeclareMathOperator{\tr}{tr}

\newtheorem{theorem}{Theorem}[section]
\newtheorem{lemma}[theorem]{Lemma}
\newtheorem{corollary}[theorem]{Corollary}
\newtheorem{proposition}[theorem]{Proposition}
\newtheorem*{definition}{Definition}

\DeclareMathOperator{\sff}{II}

\title{Ancient pancake solutions to fully nonlinear curvature flows}
\author{Mat Langford}
\author{Sathyanarayanan Rengaswami}
\date{March 2023}

\usepackage{graphicx}

\pgfplotsset{compat=1.17}
\begin{document}

\begin{abstract}
We construct $O(1)\times O(n)$-invariant ancient ``pancake'' solutions 
to a large and natural class of fully nonlinear curvature flows. We then establish that each is unique within the class of $O(n)$-invariant ancient solutions to the corresponding flow which sweep out a slab by carrying out a fine asymptotic analysis for this class. This extends the main results of \cite{BLT} to a surprisingly general class of flows.
\end{abstract}


\maketitle

\setcounter{tocdepth}{1}

\tableofcontents



\section{Introduction}

Extrinsic geometric flows drive  hypersurfaces in their normal direction with speed determined pointwise by their extrinsic curvature. The most well-known example is the mean curvature flow, but several other interesting flows have been studied in the literature, and have found applications in, for example, geometry, materials science, general relativity, and image processing. 

\emph{Ancient} solutions to nonlinear parabolic partial differential equations typically arise as blowup limits about singularities, and tend to exhibit rigidity properties
. 
Our purpose here is to analyse ``ancient pancake'' solutions (i.e. compact convex ancient solutions which are confined to slab regions) to a large and natural class of extrinsic geometric flows. In particular, we establish the following classification theorem.

\begin{theorem}\label{thm:pancakes}
Given any non-degenerate admissible speed function, there exists an $O(1)\times O(n)$-invariant ancient solution $\{\partial\Omega_t\}_{t\in(-\infty,0)}$, $\Omega_t\underset{\text{convex}}{\subset}\R^{n+1}$, to the corresponding extrinsic geometric flow exhibiting the following behaviour:
\begin{enumerate}[(a)]
\item $\{\lambda\Omega_{\lambda^{-2}t}\}_{t\in(-\infty,0)}\to\{B^{n+1}_{\sqrt{-2\phi_1t}}\}_{t\in(-\infty,0)}$ locally uniformly in the smooth topology as $\lambda\to\infty$, where $\phi_1$ is the value the speed takes on the unit sphere.
\item $\cup_{t\in(-\infty,0)}\Omega_t=\{(x,y)\in \R\times\R^{n-1}:\vert x\vert<\frac{\pi}{2}\}$, and
\item $\{\partial\Omega_{t+\tau}-P(w,\tau)\}_{t\in(-\infty,-\tau)}\to \{\Gamma_t^n\}_{t\in(-\infty,0)}$ locally uniformly in the smooth topology as $\tau\to-\infty$ for every $w\in \{0\}\times S^{n-1}$, where $P(w,t)$ is the point on $\partial\Omega_t$ whose outer unit normal is $w$, and 
\[
\Gamma^n_t\coloneqq\{p=(x,\hat x)\in \left(-\tfrac{\pi}{2},\tfrac{\pi}{2}\right)\times\R^{n}:-p\cdot w=\log\sec x+t\}
\]
is the Grim hyperplane with velocity $w$. 
\end{enumerate}
Moreover, $\{\partial\Omega_t\}_{t\in(-\infty,0)}$ is unique amongst $O(n)$-invariant convex ancient solutions to the corresponding flow satisfying property (b). 
\end{theorem}

The class of speeds we consider (see \S \ref{sec:admissible speeds} for a definition and discussion) is large and natural in view of the behaviour described in Theorem \ref{thm:pancakes}. It includes of course the mean curvature, which is uniformly elliptic, but also many examples whose ellipticity degenerates at the boundary of the positive cone (such as power means 
in the principal curvatures and linear combinations of these with highly degenerate speeds such as the $n$-th root of the Gauss curvature). 

Ancient solutions confined to a slab (the region between two parallel hyperplanes) are natural in view of Wang's slab dichotomy \cite{Wang}, which states, for convex ancient mean curvature flows $\{\partial\Omega_t\}_{t\in(-\infty,\omega)}$, that if the region $\cup_{t\in(-\infty,0)}\Omega_t$ is a strict subset of $\R^{n+1}$, then it is a slab. 
%
%
%
%
%
When $n=1$, an ancient solution which sweeps out a slab region in the plane was constructed by Bourni \emph{et al.} in \cite{MR4375790} for flows by certain powers of the curvature.
On the other hand, while it is known, for a large class of flows, that the shrinking spheres are the only ``pinched'' ancient solutions \cite{MR4129354,MR3935256}, non-spherical $O(n)$-invariant ancient solutions (to a somewhat different but nontrivially overlapping class of flows) which sweep out all of space were constructed by Risa and Sinestrari \cite{RisaSinestrari} (see also \cite{MR4234099,RisaThesis}). When a suitable differential Harnack inequality is available \cite{AndrewsHarnack}, these ``ancient ovaloid'' solutions decompose into a pair of entire ``bowl-type'' translating solitons joined by a shrinking cylinder. (A comprehensive analysis of axially symmetric translating solitons for extrinsic geometric flows, which sheds light on these phenomena, was undertaken in \cite{Rengaswami}. See also \cite{MR1625754}.)


When the speed is the mean curvature, we recover the main results of \cite{BLT}, and indeed we follow a similar approach here (though there are a number of new difficulties that arise, due to the nonlinearity of the speed function and the degeneration of ellipticity as convexity degenerates): we construct the ancient pancake by taking the limit of a family of ``very old" solutions obtained by evolving rotations of ``very old'' timeslices of the Angenent oval. A rough asymptotic analysis of arbitrary convex ancient solutions that have $O(n)$-symmetry (which exploits Andrews' differential Harnack inequality \cite{AndrewsHarnack}), shows that such solutions necessarily have the additional $O(1)$-symmetry as well. With some work, we are then able to establish a fine asymptotic expansion for the radial displacements of these solutions as $t \to -\infty$, which we are able to exploit to show that they are unique.

\section{Preliminaries}

\subsection{Admissible flow speeds}\label{sec:admissible speeds}

Given a smooth function $\phi:\Gamma\subset \R^n\to \R$, a family of hypersurfaces $\{\Sigma^n_t\}_{t\in I}$ of $\R^{n+1}$ evolves by the corresponding extrinsic curvature flow (which we call the $\phi$-\emph{flow}) if there is a smooth 1-parameter family of embeddings $F:M^n \times I \to \R^{n+1}$ such that $F(M,t)=\Sigma_t$ and
\begin{equation}\label{eq:F}
    \partial_t F=-\phi(\kappa_1,\dots,\kappa_n)\nu\,,
\end{equation}
where $\nu$ is a unit normal field and $\kappa_1\leq\ldots\leq\kappa_n$ are the principal curvatures with respect to $\nu$. 


Since we are interested in convex solutions, we require the speed function to be defined in the positive cone $\Gamma=\Gamma_+ \coloneqq \{(\kappa_1, \ldots, \kappa_n): \kappa_1,\ldots,\kappa_n > 0\}$. 

\begin{definition}
A function\footnote{The regularity of $\phi$ can be relaxed to $C^2$, so long as suitable adjustments to regularity conclusions are made in what follows.} $\phi\in C^\infty(\Gamma_+)$ is an \emph{admissible speed} if it is
\begin{enumerate}[(a)]
\item symmetric: $\phi(z_{\sigma(1)},\ldots,z_{\sigma(n)})=\phi(z_1,\ldots,z_n)$ for all permutations $\sigma$ of the set \{1,\ldots,n\};
\item elliptic: $\frac{\partial\phi}{\partial z_i}>0$ for each $i=1,\dots,n$;\label{cond:parabolic}
\item 1-homogeneous: $\phi(\lambda z_1,\ldots,\lambda z_n)=\lambda\phi(z_1,\dots,z_n)$ for every $\lambda>0$;
\item inverse-concave: the function $(r_1,\ldots,r_n)\mapsto\phi(r_1^{-1},\ldots,r_n^{-1})^{-1}$ is concave.
\end{enumerate}
We call an admissible speed $\phi$ \emph{non-degenerate} if it is
\begin{enumerate}[(e)]
\item non-degenerate: the function $s\mapsto \phi(1,s,\dots,s)$ is of class $C^2([0,\infty))$ and satisfies $\phi(1,0,\dots,0)>0$.
\end{enumerate}
\end{definition}

Note that ellipticity of a 1-homogeneous function on $\Gamma_+$ is equivalent to ellipticity on the subset $\Gamma_+\cap S^n$. Moreover, any 1-homogeneous elliptic function on $\Gamma_+$ extends continuously to $\overline{\Gamma}_+$ \cite[Lemma 1]{MR3070558} 
and is necessarily positive in $\Gamma_+$ (due to Euler's identity for homogeneous functions). Since we assume that $\phi(1,0,\ldots,0) \neq 0$, we may without loss of generality normalize the speed function so that $\phi(1,0,\ldots,0)=1$. We also set $\phi_1\coloneqq \phi(1,\ldots,1)$ and $\dot\phi_1\coloneqq \frac{d}{ds}\big\vert_{s=0}\{s\mapsto \phi(1,s,\dots,s)\}$; these three constants will play a major role in determining the asymptotic behaviour of the ancient solutions we consider.

Let us briefly discuss the purpose of conditions (a)-(e). The first two conditions are ``non-negotiable'' in that symmetry is needed to ensure that $\phi$ may be regarded as a smooth function of the second fundamental form $\sff$ (which in turn ensures that the composition of $\phi$ with the principal curvatures/second fundamental form is a smooth function on spacetime), while ellipticity is needed to ensure that \eqref{eq:F} may be interpreted as a parabolic partial differential equation. Note that, even for non-degenerate admissible speeds, the derivatives of $\phi$ with respect to the principal curvatures are allowed to degenerate at $\partial\Gamma_+$.

Homogeneity guarantees that \eqref{eq:F} is invariant under the parabolic rescaling $\Sigma_t\mapsto \lambda\Sigma_{\lambda^{-2}t}$. Note that, under the milder condition of \emph{asymptotic 1-homogeneity}, meaning that the limit
\[
(T\phi)(z)\coloneqq \lim_{\lambda\to\infty}\lambda^{-1}\phi(\lambda z)
\]
exists, blow-up limits (with respect to the rescaling $\Sigma_t\mapsto \lambda\Sigma_{\lambda^{-2}t}$) to the flow \eqref{eq:F} with speed $\phi$ evolve by \eqref{eq:F} with $\phi$ replaced by the 1-homogeneous speed $T\phi$. So our results are relevant also to certain flows by asymptotically 1-homogeneous speeds.

Inverse-concavity is typically invoked to guarantee the applicability of the Krylov--Safanov Harnack inequality; this is not needed here, however, due to the symmetry of the solutions we consider --- instead, we require it in order to make use of Andrews' differential Harnack inequality, which guarantees, e.g., that ancient solutions are asymptotically modelled on translating solitons.

The non-degeneracy condition ensures that $\phi$ is uniformly elliptic in the principal direction corresponding to the largest principal curvature, even as the other principal curvatures approach zero (the derivatives of $\phi$ in the remaining directions are still allowed to degenerate, however). Under the normalization $\phi(1,0,\dots,0)=1$, the non-degeneracy condition guarantees that the cylinder $\{\Gamma_t\times \R^{n-1}\}_{t\in I}$ is a solution to \eqref{eq:F} whenever $\{\Gamma_t\}_{t\in I}$ is a solution to curve shortening flow. This ensures that the ancient solutions we construct are asymptotically modelled on Grim hyperplanes.


\subsubsection{Examples}
The class of admissible speeds is very large (see, for example, \cite[\S18.3]{EGF}). As mentioned in the introduction, the power means
\[
P_r(z_1,\dots,z_1)\coloneqq\left(\sum_{i=1}^nz_i^r\right)^{\frac{1}{r}}\,,\;\; r>0
\]
give rise to non-degenerate admissible speeds. Further examples are given by convex functions $f\in C^\infty(\Gamma)$ satisfying conditions (a)-(c) with $\Gamma_+\cap S^n\Subset \Gamma$. 

Note also that the class of (non-degenerate) admissible speeds is closed under natural compositions: if $\zeta_1,\dots,\zeta_k\in C^\infty(\Gamma_+^n)$ and $\phi\in C^\infty(\Gamma_+^k)$ are (non-degenerate) admissible speeds, then so is $\phi(\zeta_1,\dots,\zeta_k)$. 



\subsection{The Angenent oval}

The Angenent oval (a.k.a the paperclip) is a solution to curve shortening flow that lies in the strip $\left( -\frac{\pi}{2},\frac{\pi}{2} \right)\times \R$. It is defined for $t \in (-\infty,0)$ and satisfies the implicit equation
\begin{equation}
    \cos x = e^t \cosh{y}\,.
\end{equation}

In the turning angle parametrization $(x(\theta,t),y(\theta,t))$, where $\theta$ is the angle made by the tangent vector and the $x$-axis when the curve is oriented counterclockwise, the $x,y$ coordinates are given explicitly by
\begin{equation*}
    x(\theta,t)= \arctan \left( \frac{\sin\theta}{\sqrt{\cos^2\theta+a^2(t)}}\right)
\end{equation*}
and
\begin{equation*}
    y(\theta,t)= -t + \log \left( \frac{\sqrt{\cos^2\theta+a^2(t)}-\cos \theta}{\sqrt{a^2(t)+1}}\right)\,,
\end{equation*}
where $a^2(t) \coloneqq \frac{1}{e^{-2t}-1}$.

We also define the horizontal and vertical displacements $h(t)$ and $\ell(t)$ by
\begin{equation*}
    h(t) \coloneqq x(\tfrac{\pi}{2},t)
\end{equation*}
and 
\begin{equation*}
    h(t) \coloneqq y(\pi,t)\,.
\end{equation*}
These displacements satisfy the estimates
\begin{eqnarray} \label{DispEst}
    \frac{\pi}{2}(1-e^{-t}) \leq h(t) \leq \frac{\pi}{2}\,,\\
    -t \leq \ell(t) \leq -t+ \log 2\,.
\end{eqnarray}

\subsection{$O(n)$-invariance}

Let us give $\R^{n+1}$ the coordinates $(x,y,z)\in \R\times\R\times\R^{n-1}$. By $O(n)$-invariance of a hypersurface $\Sigma^n \subset \R^{n+1}$, we will mean that it is invariant under the action of $O(n)$ on the $(y,z)$ hyperplane; 
this of course means that $\Sigma^n=\{xe_1+y\eta : xe_1+ye_2\in \Sigma^n\cap \E^2, \eta \in S^n\cap\{0\}\times\R^{n}\}$, where $\E^2\coloneqq \R^2\times\{0\}\subset\R^{n+1}$. We will refer to the curve $\Sigma^n\cap \E^2$ as the \emph{profile curve} of $\Sigma^n$.


Conversely, given a smooth  convex curve $\Gamma\subset \R^2$ that is reflection symmetric about the $x$-axis, we can form an $O(n)$-invariant hypersurface by revolving $\Gamma\subset \R^2\times\R^{n-1}$ about the $x$-axis. This hypersurface has only two distinct principal curvatures, $\kappa$, the curvature of $\Gamma$ (multiplicity one), and $\lambda$, the ``rotational'' curvature (multiplicity $n-1$), which is given by
\[\lambda=
\begin{cases}
    \frac{\cos\theta}{y}, & \theta\neq \frac{\pi}{2},\frac{3\pi}{2}\\
    \kappa, & \theta=  \frac{\pi}{2},\frac{3\pi}{2}\,.
\end{cases}\]
Note that $\lambda$ is a smooth function on $S^1=\R/2\pi\Z$ and satisfies
\begin{equation}\label{lambdatheta}
    \kappa\lambda_\theta=-\lambda\tan\theta(\kappa-\lambda)\,.
\end{equation}

\subsubsection{Notation} 

Given an $O(n)$-invariant hypersurface with profile curvature $\kappa$ and rotational curvatures $\lambda$, we shall, by an abuse of notation, write $\phi(\kappa,\lambda)$ to mean $ \phi(\kappa,\lambda,\ldots,\lambda)$. On an $O(n)$-invariant flow with profile curvature $\kappa(\theta,t)$ and rotational curvatures $\lambda(\theta,t)$, we also sometimes use $\phi(\theta,t)$ to mean $\phi(\kappa(\theta,t),\lambda(\theta,t))$; the distinction will be clear from context, however. 
By $\phi_t$ and $\phi_\theta$, etc, we mean the $t$ and $\theta$ derivatives, respectively, of this function. By $\phi_\kappa$ and $\phi_\lambda$, etc, we mean the derivatives of the two-variable function $\phi(\kappa,\lambda)$ with respect to its first and second arguments, respectively.

\subsection{$O(n)$-invariance and curvature flows}

We assume in this section that $\phi\in C^\infty(\Gamma_+)$ is an admissible (but not necessarily non-degenerate) speed function.

We start by computing the evolution equations for $\kappa,\lambda,\phi$ and the enclosed area $A$.

\begin{lemma} \label{evoeq}
Along any $O(n)$-invariant solution to the flow \eqref{eq:F},
\begin{align}\label{eq:evolve kappa}
\kappa_t={}& \kappa^2 \phi_\kappa \kappa_{\theta \theta}+ \kappa^2 (\phi_{\kappa\kappa}\kappa_\theta^2 + 2\phi_{\kappa \lambda} \kappa_\theta \lambda_\theta + \phi_{\lambda \lambda} \lambda_\theta^2)\nonumber\\
{}&+\kappa \phi_\lambda (-\kappa_\theta \lambda_\theta-\lambda\tan\theta(\kappa_\theta-2\lambda_\theta))+\kappa(\phi_\kappa\kappa^2 + \phi_\lambda \lambda^2)\,,
\end{align}
\begin{equation}\label{eq:evolve lambda}\lambda_t= \kappa^2 \phi_\kappa \lambda_{\theta\theta}-\lambda(\phi+\phi_\kappa\kappa)\tan\theta \lambda_\theta +\lambda (\phi_\kappa \kappa^2 + \phi_\lambda \lambda^2)\,,
\end{equation}
\begin{equation}\label{eq:evolve phi}\phi_t= \phi_\kappa\kappa^2 \phi_{\theta\theta}-\phi_\lambda\lambda^2 \tan\theta \phi_\theta + \phi(\phi_\kappa \kappa^2 + \phi_\lambda \lambda^2)\,,
\end{equation}
and
\begin{equation}\label{eq:evolve A}
-\frac{dA}{dt}=\int_0^{2\pi} \phi\,d\theta\,.
\end{equation}
\end{lemma}
\begin{proof}
With respect to the Gauss map parametrization, the support function $\sigma(\cdot,t):S^n\to\R$ of the solution evolves as
\[\partial_t \sigma=-\phi\]
and 
we can express the principal curvatures $\kappa,\lambda$ 
as
\begin{align*}
    \kappa^{-1}&=\sigma_{\theta\theta}+\sigma\\
    \lambda^{-1}&=\sigma-\sigma_\theta \tan\theta\,.
\end{align*}

Differentiating with respect to $t$ gives
\begin{align*}
   \kappa_t&=\kappa^2(\phi_{\theta\theta}+\phi)\,,\\
   \lambda_t&=\lambda^2(\phi-\phi_\theta \tan\theta).
\end{align*}
Formulae \eqref{eq:evolve kappa}, \eqref{eq:evolve lambda} are obtained now by using the chain rule and relating $\lambda_{\theta\theta}$ to lower order terms by differentiating \eqref{lambdatheta}.

The evolution equation for $\phi$ follows from those for $\kappa$ and $\lambda$ since
\[
\phi_t=\phi_\kappa\kappa_t+\phi_\lambda\lambda_t\,.
\]

The identity \eqref{eq:evolve A} is just the usual first variation of enclosed area.
\end{proof}

The following fundamental lemmas apply to all $O(n)$-invariant solutions.

\begin{lemma}\label{lem:KappaOverLambda}
The ratio $u \coloneqq \kappa / \lambda$ evolves under \eqref{eq:F} according to
\begin{align}\label{eq:KappaOverLambda}
u_t-{}&\kappa^2 \phi_\kappa u_{\theta \theta}-2 \phi_\kappa \kappa^2 (\lambda_\theta/\lambda)u_\theta\nonumber\\
{}&= \lambda^{-1}D^2\phi((\kappa_\theta,\lambda_\theta),(\kappa_\theta,\lambda_\theta))-\phi_\lambda \lambda^2 \tan\theta u_\theta -2\phi \tan^2\theta (\kappa-\lambda)\,.
\end{align}
\end{lemma}
\begin{proof}
This follows by direct calculation.
\end{proof}

\begin{lemma} \label{lem:u crit}
At a critical point of $u$, $D^2\phi((\kappa_\theta,\lambda_\theta),(\kappa_\theta,\lambda_\theta))=0$.
\end{lemma}
\begin{proof}
At a critical point of $u$, we have 
\begin{align*}
&u_\theta=0\\ 
\implies &\lambda \kappa_\theta-\kappa \lambda_\theta=0\\
\implies &\kappa_\theta/\lambda_\theta=\kappa/\lambda\,,
\end{align*}
which means that $(\kappa_\theta,\lambda_\theta)$ is a multiple of $(\kappa,\lambda)$. Now since $\phi$ is 1-homogeneous, we have that $D^2\phi((\kappa,\lambda),(\kappa,\lambda))=0$ and the claim follows.
\end{proof}

\begin{corollary}\label{cor:sturm}
The inequality $\kappa\ge\lambda$ is preserved along $O(n)$-invariant solutions to \eqref{eq:F}.
\end{corollary}
\begin{proof}
Given $\varepsilon>0$, consider the function $u^\varepsilon\coloneqq u-1+\varepsilon\mathrm{e}^{t}$, where $u\coloneqq \kappa/\lambda$. By hypothesis, $u^{\varepsilon}(\cdot,0)\ge \varepsilon>0$. We claim that $u_{\varepsilon}$ remains positive for all positive times. Suppose, to the contrary, that there is some $t_0>0$ and some $\theta_0\in S^1$ such that $u^{\varepsilon}(\theta_0,t_0)=0$ but $\min_{S^1}u^{\varepsilon}>0$ for all $t<t_0$. In particular, 
\[
0\ge u^\varepsilon_{t}\,,\;\ 0\ge -u^\varepsilon_{\theta\theta}\,,\;\;\text{and}\;\; 0=u^\varepsilon_{\theta}=u_{\theta}\,,
\]
at $(\theta_0,t_0)$, and hence, by \eqref{eq:KappaOverLambda} and Lemma \ref{lem:u crit},
\begin{align*}
0\ge{}&u^\varepsilon_t-\kappa^2 \phi_\kappa u^\varepsilon_{\theta \theta}-2 \phi_\kappa \kappa^2 (\lambda_\theta/\lambda)u^\varepsilon_\theta\\
={}&\varepsilon\mathrm{e}^{t_0}+2\varepsilon\mathrm{e}^{t_0}\phi\lambda\tan^2\theta\\
>{}&0
\end{align*}
at $(\theta_0,t_0)$, yielding a contradiction. So $u^\varepsilon$ remains positive and hence, taking $\varepsilon\to 0$, we conclude that $u-1$ remains non-negative.
\end{proof}

\begin{corollary} \label{lambdanondec}
    $\lambda(\cdot,t)$ is nondecreasing on $[\frac{\pi}{2},\pi]$ for each $t$.
\end{corollary}
\begin{proof}
    This is a combination of (\ref{lambdatheta}) and Corollary \ref{cor:sturm}
\end{proof}

\begin{corollary}\label{cor:pinching preserved}
The inequality $\kappa/\lambda\le C$ is preserved along $O(n)$-invariant solutions to \eqref{eq:F} which satisfy $\kappa\ge\lambda$.
\end{corollary}
\begin{proof}
Since the inequality $\kappa\ge\lambda$ ensures that the reaction term in \eqref{eq:KappaOverLambda} has the correct sign for preserving upper bounds, the claim follows from the maximum principle in a similar manner to Corollary \ref{cor:sturm}.
\end{proof}

A well-known argument originally due to Chou (formerly Tso) \cite{MR812353} provides an estimate for the speed for as long as the inradius remains bounded from below.
\begin{lemma}\label{lem:Chou}
If the support function $\sigma:S^n\times[0,t_0]\to\R$ of a solution $\{M_t\}_{t\in[0,t_0]}$ to \eqref{eq:F} satisfies
\[
2r\le\sigma\le R
\]
for all $t\in[0,t_0]$, then
\begin{equation}\label{eq:Chou estimate}
\phi\le \frac{R}{r}\max\left\{2\phi_1r^{-1},\max_{M_0}\phi\right\}
\end{equation}
and
\begin{equation}\label{eq:Chou interior estimate}
\phi\le C\left(2r^{-1}+t^{-\frac{1}{2}}\right)\,,
\end{equation}
where $C=C(n,\phi_1,\frac{R}{r})$ and we recall that $\phi_1\coloneqq \phi(1,\dots,1)$.
\end{lemma}
\begin{proof}
If we define the function $\phi_\ast:\Gamma_+\to\R$ by
\[
\phi_\ast(r)\coloneqq -\phi(r^{-1})\,,
\]
then, with respect to the Gauss map parametrization,
\[
\partial_t\sigma=-\phi=\phi_\ast(\rho_1,\dots,\rho_n)\,,
\]
where $\rho_i\coloneqq\kappa_i^{-1}$ are the principal radii. Note that, under the Gauss map parametrization, the principal radii are the eigenvalues of the tensor
\[
A\coloneqq \overline\nabla{}^2\sigma+\sigma \overline g\,,
\]
where $\overline g$ and $\overline\nabla$ are the standard metric and connection on $S^n$.

Consider the function $v\coloneqq\frac{\phi}{\sigma-r}$. If we denote by $\dot\phi^{ij}$ and $\dot\phi^i$ the derivatives of $\phi$ with respect to $\sff_{ij}$ and $\kappa_i$, respectively, and by $\dot\phi_\ast^{ij}$ and $\dot\phi_\ast^i$ the derivatives of $\phi_\ast$ with respect to $A_{ij}$ and $\rho_i$, respectively, then, at a new interior maximum of $v$, we find that
\begin{align*}
0\le{}& (\partial_t-\dot\phi_\ast^{ij}\overline\nabla_i\overline\nabla_j)\frac{\phi}{\sigma-r}\\
={}&\frac{\phi}{\sigma-r}\left(\frac{(\partial_t-\dot\phi_\ast^{ij}\overline\nabla_i\overline\nabla_j)\phi}{\phi}-\frac{(\partial_t-\dot\phi_\ast^{ij}\overline\nabla_i\overline\nabla_j)\sigma}{\sigma-r}\right)+2\dot\phi_\ast^{ij}\overline\nabla_i\frac{\phi}{\sigma-r}\overline\nabla_j\sigma\\
={}&\frac{\phi}{\sigma-r}\left(\tr(\dot\phi_\ast)+\frac{\phi+\dot\phi_\ast(A)-\sigma\tr(\dot\phi_\ast)}{\sigma-r}\right)\\
={}&\frac{\phi}{\sigma-r}\left(\dot\phi^{i}\kappa^2_i+\frac{2\phi-\dot\phi^{i}\kappa^2_i\sigma}{\sigma-r}\right)\,.
\end{align*}
That is,
\begin{align*}
\frac{\dot\phi^{i}\kappa^2_i}{\phi^2}\frac{\phi}{\sigma-r}\le\frac{2r^{-1}}{\sigma-r}\le 2r^{-2}\,.
\end{align*}
Since $\phi$ is inverse-concave, the first claim now follows from \cite[Lemma 5]{MR3070558}. 

To obtain the second claim, consider instead the ratio $\frac{t\phi}{\sigma-r}$.
\end{proof}


\begin{proposition}\label{prop:round point}
Let $\Sigma$ be an $O(n)$-invariant bounded convex hypersurface of $\R^{n+1}$. If $\kappa\ge\lambda$ on $\Sigma$, then (for any admissible speed $\phi$) the solution to \eqref{eq:F} starting from $\Sigma$ contracts to a round point in finite time.
\end{proposition}
\begin{proof}
Since the profile curve $\gamma:S^1\times [0,T)\to\R^2$ of the solution satisfies the equation
\[
\left\langle\gamma_t(\theta,t),\nu(\theta,t)\right\rangle=-f(\theta,t,\kappa(\theta,t))\,,
\]
where
\[
f(\theta,t,k)\coloneqq \phi(k,\lambda(\theta,t))\,,
\]
the support function satisfies an inhomogeneous parabolic equation with smooth coefficients. By Corollary \ref{cor:pinching preserved}, this equation remains uniformly parabolic, so Lemma \ref{lem:Chou} and estimates for inhomogeneous uniformly parabolic equations in one space variable \cite[\S XI.6]{Lieberman} provide \emph{a priori} estimates in $C^{k,\alpha}$ for every $k$, depending only on $k$, $f$, $\max_{S^1\times\{0\}}{\kappa/\lambda}$ and $r_0$, on any time interval $[0,t_0]$ on which $\sigma$ is bounded from below by $r_0$. Moreover, by Corollary \ref{cor:pinching preserved} and Andrews' lemma \cite[Theorem 5.1]{Andrews94}, the ratio of maximum to minimum width of the profile curve remains uniformly bounded throughout the evolution. A standard blow-up argument (see \cite{Andrews94}) in conjunction with \eqref{eq:KappaOverLambda} now implies the claim.
\end{proof}

We note that Proposition \ref{prop:round point} follows from a more general theorem of McCoy--Mofarreh--Wheeler \cite[Theorem 6.1]{MR3338439} (cf. \cite[Theorem 1]{RisaSinestrari}). It may also be seen as a consequence of a result of Andrews and McCoy \cite{MR3528528} (cf. \cite{MR2729317}).

\section{Existence}

We shall construct our ancient pancake solutions by taking the limit of a sequence of old-but-not-ancient solutions. As in \cite{BLT}, suitable old-but-not-ancient solutions are obtained by evolving rotated timeslices of the Angenent oval.

\subsection{The approximating solutions}

Let $\gamma:S^1 \times (-\infty,0) \to \R^2$ be the turning angle parametrization of the Angenent oval and set $\Gamma_t \coloneqq \gamma(S^1,t)$. For each $R<0$, we define $\Sigma^R$  to be the hypersurface obtained by revolving $\Gamma_{-R}$ about the $x$-axis, i.e. 
\begin{equation*}
    \Sigma^R \coloneqq \{x_R(\theta)e_1+y_R(\theta)\varphi: \theta\in S^1, \varphi\in S^{n-1}\subset\{e_1\}^{\perp}\}\,,
\end{equation*}
where $x_R,y_R$ are the coordinate functions for $\Gamma_{-R}$, i.e., $\gamma(\theta,-R)=(x_R(\theta),y_R(\theta))$. Let $F_0^R:S^n\to \R^{n+1}$ be the Gauss map parametrization of $\Sigma_R$.


Consider the maximal $O(n)$-invariant solution $F_R:S^n \times [-T_R,0) \to \R^{n+1}$ to the $\phi$-flow with initial data $F_R(\cdot,0)=F_0^R(\cdot)$. Since $\Sigma^R$ satisfies $\kappa\ge\lambda$ (see the proof of \cite[Lemma 4.1]{BLT}), Corollary \ref{cor:sturm} ensures that this inequality continues to hold for $t>-T_R$. Proposition \ref{prop:round point} then guarantees that the timeslices $\Sigma_t^R\coloneqq F_R(S^n,t)$ shrink to a round point as $t\to 0$.

Define $\Sigma_t^R \coloneqq F_R(S^n,t)$ and $\Gamma_t^R \coloneqq \Sigma_t^R \cap \E^2$. By rotating as necessary, we may assume that the unit tangent vector field $\tau_R$ satisfies $\tau_R(0)=e_1$. Due to to uniqueness of solutions, the reflection symmetry about the two axes of $\E^2$ are preserved under the flow. Consequently, the points $\gamma_R(\pi/2,t),\gamma_R(\pi,t)$ are the unique points of $\Gamma_t^R$ that lie on the positive axes. Their distances from the origin, $h_R(t) \coloneqq x(\gamma_R(\pi/2,t))$ and $\ell_R(t) \coloneqq y(\gamma_R(\pi,t))$, are referred to as the horizontal and vertical displacements of $\Gamma_t^R$ and they play an important role in our analysis.

We prove the existence of an ancient pancake solution by considering the family of flows $\{F_R\}_{R>0}$ and and taking a (subsequential) limit as $R \to \infty$. For this, we need estimates on the displacements $|F_R|$ and the second fundamental form $\sff$ that are uniform for sufficiently large $R$, so that the Arzel\`{a}-Ascoli theorem gives us the existence of a convergent subsequence. 

\subsection{Displacement and curvature estimates}
For estimates on $|F_R|$, we take advantage of the fact that the flow preserves convexity. Thus it will suffice to have estimates on the horizontal and vertical displacements $h_R(t) \coloneqq x(\gamma_R(\pi/2,t))=\langle\gamma_R(\pi/2,t),e_1\rangle$ and $\ell_R(t) \coloneqq y(\gamma_R(\pi,t))=\langle\gamma_R(\pi,t),e_2\rangle$.


\begin{lemma}\label{lem:h ell  estimates}
    $\ell_R(t) \geq h_R(t)$
\end{lemma}
\begin{proof}
    Due to Corollary \ref{cor:sturm},
   \begin{align*}
    -\ell'(t)&= \phi(\kappa(\pi,t),\lambda(\pi,t))\\
    &\geq \phi(\lambda(\pi,t),\lambda(\pi,t))\\
    &\geq \phi(\lambda(\pi/2,t),\lambda(\pi/2,t))\\
    &=\phi(\kappa(\pi/2,t),\lambda(\pi/2,t))\\
    &=-h'(t)\,.
\end{align*}
The claim follows by integrating the above inequality from $t$ to $0$, at which time we know that $\ell(0)=h(0)=0$ due to Proposition \ref{prop:round point}. 
\end{proof}

Recall that we define the constant $\phi_1 \coloneqq \phi(1,1)$.

\begin{lemma}\label{lem:hl}
$-\frac{\pi}{2}t \leq h_R(t)\ell_R(t) \leq -\pi\phi_1 t$.
\end{lemma}
\begin{proof}
Due to Corollary \ref{cor:sturm}, we have that
\begin{equation}
\kappa=\phi(\kappa,0)\leq\phi(\kappa,\lambda)\leq\phi(\kappa,\kappa)=\kappa\phi_1\,.
\end{equation}

Integrating \ref{eq:evolve A} and using the above estimate, we get
\begin{equation}
-2\pi t \leq A_R(t) \leq -2\pi\phi_1 t\,.
\end{equation}
Lastly, by convexity of our solution, we may estimate the area from the inside and outside by the area of quadrilaterals through the points of maximal horizontal and vertical displacements to get
\begin{equation}
2h_R(t)\ell_R(t) \leq A_R(t) \leq 4h_R(t)\ell_R(t)\,.
\end{equation}
The desired estimate for $h_R \ell_R$ now follows.
\end{proof}

\begin{lemma}
$\frac{R}{2\phi_1}(1-e^{-R)} \leq T_R \leq (R+\log2)$
\end{lemma}
\begin{proof}
By estimates in (\ref{DispEst}) for Angenent ovals,
\begin{eqnarray}
\frac{\pi}{2}\left(1-e^{-R}\right) \leq h_R(-T_R) \leq \pi/2 \\
R \leq  \ell_R(-T_R) \leq R+ \log2
\end{eqnarray}
Now multiplying these inequalities together and using Lemma \ref{lem:hl} yields the desired inequality.
\end{proof}

\begin{lemma}\label{lem:phi min est}
$\phi^R_{\min}(t)= \phi^R(\tfrac{\pi}{2},t)\leq \frac{2h_R}{h_R^2+\ell_R^2}\phi_1$. 
\end{lemma}
\begin{proof}
    A circle $\mathcal{C}$ centered on the $x$-axis and passing through $\gamma_R(\pi/2,t)$ and $\gamma_R(\pi,t)$ has radius $r=\frac{\ell_R^2+h_R^2}{2h_R}$ and center $(-(r-h_R),0)$, and hence any circle with radius $\rho<\frac{\ell_R^2+h_R^2}{2h_R}$ and center $(-(\rho-h_R),0)$ is tangent to $\Gamma_t$ at $\gamma(\pi/2,t)$ and the point $\gamma_R(\pi,t)$ lies to its outside. Due to the curvature of $\Gamma_t$ being minimized at $\gamma_R(\pi/2,t)$, it then follows that  $\kappa_R(\pi/2,t) \leq \frac{2h_R}{\ell_R^2+h_R^2}$ (See \cite[Claim 4.4.1]{BLT}.) This yields
    \[\phi^R\left(\tfrac{\pi}{2},t\right) = \phi(\kappa_R(\pi/2,t),\lambda_R(\pi/2,t))=\phi(\kappa_R(\pi/2,t),\kappa_R(\pi/2,t)) =\phi_1\kappa_R(\pi/2,t) \, ,\]
    from which the claim follows.
\end{proof}

\begin{lemma}\label{hEst}
$h_R(t) \geq \frac{\pi}{2}(1-e^{-R}) \exp\left(\frac{2\phi_1}{t}\right)$
\end{lemma}
\begin{proof}
Since $h_R\le\frac{\pi}{2}$, Lemmas \ref{lem:hl} and \ref{lem:phi min est} yield
\[
-h_R'(t)=\phi^R(\tfrac{\pi}{2},t)\leq \frac{2\phi_1h_R}{\ell(t)^2} \leq \frac{2\phi_1h_R}{(-t)^2}\,.
\]
Integrating from $T_R$ to $-t$, we obtain
\[
h_R(t) \geq h_R(T_R)\exp\left(\frac{2\phi_1}{T_R}+\frac{2\phi_1}{t}\right).
\]
The claim now follows from the estimate (\ref{DispEst}).
\end{proof}

\begin{lemma} \label{lEst}
$\ell_R(t)\leq \frac{-2\phi_1t}{1-\mathrm{e}^{-R}}\exp\left(\frac{2\phi_1}{-t}\right)$.
\end{lemma}
\begin{proof}
This follows by combining the estimates of Lemmas \ref{lem:hl} and \ref{hEst}.
\end{proof}

\subsection{The existence argument}


Combining the above estimates, we may now extract a limit along some sequence $R_j\to-\infty$.

\begin{theorem}
Given any non-degenerate admissible speed $\phi$, there is a sequence of approximating solutions which converges locally uniformly in the smooth topology to a smooth, compact, convex, locally uniformly convex, $O(1) \times O(n)$-invariant ancient solution to the $\phi$-flow which lies in the slab $\Omega \coloneqq \{(x,y,z) \in \R \times \R \times \R^{n-1}: |x| < \frac{\pi}{2}\}$ and in no smaller slab.
\end{theorem}
\begin{proof}
We need uniform-in-$R$ estimates for $F_R$ and its derivatives on compact subsets of time. A bound for $\vert F_R\vert$ follows from Lemma \ref{lEst}. A bound for the curvature follows from Lemma \ref{lem:Chou} due to the inradius bound (which is a consequence of Lemmas \ref{lem:h ell  estimates} and \ref{hEst} and convexity) and the circumradius bound (which is a consequence of Lemma \ref{lEst}). Since $\phi_\kappa$ is bounded uniformly from above and below, higher order estimates then follow from estimates for inhomogeneous uniformly parabolic equations in one space variable as in Proposition \ref{prop:round point}. Well-known arguments then provide uniform-in-$R$ estimates for $F_R$ in $C^k$ for all $k$, and also a lower bound for $DF_R$. 

Thus, by the Arzel\`{a}--Ascoli theorem, there exists a sequence of flows $F_{R_j}$, with $R_j \to -\infty$, converging locally uniformly in the smooth topology to a smooth limit flow. The limit is certainly weakly convex since this is the case for the approximating solutions. Moreover, the upper bound for $\ell_R$ ensures that the limit is compact (and in particular does not degenerate into two parallel hyperplanes). Strict convexity then follows from applying the strong maximum principle to \eqref{eq:evolve kappa} (under slightly stronger conditions on the speed $\phi$, we could also apply \cite[Proposition A.2]{Lynch}). The limit solution inherits rotational and reflection symmetries from the approximating solutions. Convexity and the lower bound for $h_R$ ensure that the solution lies in no smaller slab.
\end{proof}

\section{Unique asymptotics and reflection symmetry}

We wish to show in Section \ref{sec:uniqueness} that the solution constructed above is unique in the class of $O(n)$-invariant convex ancient solutions in the slab $[-\frac{\pi}{2},\frac{\pi}{2}]\times \R^n$. We will use Alexandrov's moving plane method to show this. But to be able to do this, we need to know the asymptotics of any such solution; that is the subject of this section. 

Consider any convex ancient solution $F:S^n \times (-\infty,0) \to \R^{n+1}$ of $\phi$-flow that is $O(n)$-invariant with respect to some $e_1 \in S^n$ and lies in the slab $\Omega \coloneqq \{(x,y,z) \in\R\times\R\times\R^{n-1}:|x|< \pi/2\}$ and in no smaller slab. Define $\Sigma_t=\Sigma_t^n \coloneqq F(S^n,t)$. Given some unit vector $e_2 \perp e_1 $, define the plane $\E^2\coloneqq \spa\{e_1,e_2\}$, and parametrize the profile curve $\Gamma_t \coloneqq \E^2 \cap \Sigma_t$ with respect to turning angle by a curve $\gamma:S^1 \times (-\infty,0) \to \E^2$. Denote, as usual, the unit tangent and unit normal by $\tau, \nu$ respectively and the curvature and rotational curvature by $\kappa(\cdot,t), \lambda(\cdot,t)$  respectively. We continue to use $\phi(\cdot,t) \coloneqq \phi(\kappa(\cdot,t),\lambda(\cdot,t))$.

Finally, for each $t<0$ we define the vertical dispacement to be 
\[
\ell(t) \coloneqq \max_{\theta\in S^1} |\langle \gamma(\theta,t),e_2 \rangle|.
\]
Due to convexity and rotational symmetry,
\[
\ell(t)=\langle \gamma(\pi,t),e_2 \rangle=-\langle \gamma(0,t),e_2 \rangle.
\]
Note that we do not know \emph{a priori} that the curve is reflection-symmetric about the $y$-axis.

\subsection{Unique asymptotics}

The following two lemmas describe the shape of arbitrary solutions $\Sigma_t$ as $t\to\infty$. They are analogous \cite[Lemmas 5.1 and 5.2]{BLT}. The proofs are also identical to theirs, since they only rely on convexity, the strong maximum principle and the Harnack inequality, all of which hold in our case.

\begin{lemma} \label{parabolic}
For any sequence of times $t_i \to -\infty$ the sequence of flows $\Sigma_t^i \coloneqq \Sigma_{t+t_i}$ defined on $(-\infty,-t_i)$ converge locally uniformly in the smooth topology to the stationary $\phi$-flow defined by $\partial \Omega$.
\end{lemma}

Now recall that we have defined a normalization on our speed such that $\phi(1,0)=1$. Define $\alpha$ by
\begin{equation}
    \alpha^{-1}\coloneqq \lim_{t \to -\infty
}\phi(\pi,t)
\end{equation}
This limit is well-defined as $\phi$ is nondecreasing in $t$ due to the differential Harnack inequality \cite{AndrewsHarnack}. Let $P$ be the inverse of the Gauss map, i.e. for any unit vector $v$, $P(v,t)$ is the point on $\Sigma_t$ with normal $v$.

\begin{lemma} \label{tip}
For any sequence of times $t_i \to -\infty$ and any unit vector $v \perp e_1$ the sequence of flows $\Sigma_t^i \coloneqq \Sigma_{t+t_i}-P(v,t_i)$ converges locally uniformly in the smooth topology to the scaled Grim hyperplane $\alpha G^n_{\alpha^{-2}t}$ where $G^n_t$ defined by
\begin{equation}
    G^n_t \coloneqq \{\theta e_1 +(t-\log \cos \theta)v\}, \, t \in (-\infty,\infty)
\end{equation}
is the standard Grim hyperplane that translates in direction $v$ with unit speed.
\end{lemma}

Due to containment of the flow in the strip, it is clear that $\alpha \leq 1$. Now we show that $\alpha=1$. We obtain this from purely geometrical considerations by showing that if $\alpha<1$,  then the area $A(t)$ contained inside $\Gamma_t$ decreases too quickly (cf. \cite{MR4127403}).

\begin{lemma}\label{lemma:lEstArbit}
$\ell(t) \geq \alpha^{-1}t$.
\end{lemma}
\begin{proof}
It is a simple consequence of the fact that $-\ell'(t)=\phi(\pi,t)\geq \alpha^{-1}$ due to the Harnack inequality.
\end{proof}

\begin{lemma}\label{lambdaleqkappa}
$\lambda(\cdot,t) \leq \kappa(\cdot,t)$.
\end{lemma}
\begin{proof}
By rotational symmetry, $\lambda( \pm \pi/2,t)=\kappa( \pm \pi/2,t)$ and by Lemma \ref{tip}, 
\[\displaystyle \lim_{t\to -\infty} \frac{\kappa}{\lambda}(\theta,t)=\infty \]
for $\theta \in (-\frac{\pi}{2},\pi/2)$. The claim now follows from Lemmas \ref{lem:KappaOverLambda} and \ref{cor:sturm}  and the symmetry of $\gamma$.
\end{proof}

\begin{lemma}\label{lem:crudeareadecay}
$\displaystyle \lim_{t\to -\infty}A'(t)=-2\pi$.
\end{lemma}
\begin{proof}
We have
\begin{align*}
    |A'+2\pi|&=\left\vert-\int_{\Gamma_t}\phi\,ds+\int_{\Gamma_t}\kappa\,ds\right\vert\\
    &= \left\vert\int_{S^1}(\phi(\kappa,\lambda)-\kappa) \frac{d\theta}{\kappa}\right\vert\\
    &=\left\vert\int_{S^1}\left(\phi(1,\lambda/\kappa)-1\right) d\theta\right\vert
\end{align*}
By Lemma \ref{lambdaleqkappa}, $0\leq\phi(1,\lambda/\kappa)-1 \leq \phi(1,1)-1$ and $\lim_{t\to-\infty}\lambda/\kappa=0$ a.e. on $S^1$, so that the integrand tends to zero a.e. and the claim follows by Lebesgue dominated convergence.
\end{proof}

\begin{corollary}
    The width of the limiting Grim hyperplane is maximal, i.e. $\alpha=1$.
\end{corollary}
\begin{proof}
We claim that for any $\varepsilon>0$, there exists $t_\varepsilon<0$ such that for any $t<t_\varepsilon$,
\[
A(t)\leq -(2\pi +\varepsilon)t
\]
Indeed, by Lemma \ref{lem:crudeareadecay}, given any $\varepsilon>0$ there exists $t_\varepsilon'<0$ such that for all $t<t_\varepsilon'$
\begin{align*}
A(t)\le{}&A(t_\varepsilon)+(2\pi+\tfrac{\varepsilon}{2})(t_\varepsilon'-t)\\
={}&-(2\pi +\tfrac{\varepsilon}{2})t+\frac{A(t_\varepsilon')+(2\pi+\tfrac{\varepsilon}{2})t_\varepsilon'}{t}t\,.
\end{align*}
so that we may set $t_\varepsilon\doteqdot \min\{t_\varepsilon',-\frac{2}{\varepsilon}(A(t_\varepsilon')+(2\pi+\tfrac{\varepsilon}{2})t_\varepsilon')\}$.

We bound the area from below by an enclosed trapezoid that we describe below. Let $C(t),D(t)\in \Gamma_t$ such that $y(C(t))=y(D(t))=0$. Without loss of generality, assume that $x(C(t))>x(D(t))$. Let $p(t)=\gamma(0,t)$ be the ``tip" of $\Gamma_t$. Now for any $\delta\in (0,1)$ we can find $t_\delta < 0 $ such that for all $t<t_\delta$, 
\[\pi-\delta \leq x(C(t))-x(D(t)) \leq \pi.\]
Since the tip region converges locally uniformly to the scaled Grim Reaper $\alpha G$, one can, (by taking $t_\delta$ more negative if need be), also find points $p^{\pm}(t)$ on $\Gamma_t$ and a constant $C_\delta$ such that
\[
y(p^-(t))=y(p^+(t))\,,
\]
\[
\pi\alpha-\delta \leq x(p^+(t))-x(p^-(t)) \leq \pi \alpha\,,
\]
and
\[
0 \leq y(p(t))-y(p^\pm(t))=\ell(t)-y(p^\pm(t)) \leq C_\delta\,.
\]
Now since the trapezoid formed by $C(t),D(t),p^\pm(t)$ is contained inside $\Gamma_t$ (by convexity), we estimate $\ell(t)$ using Lemma \ref{lemma:lEstArbit} to get 
\[(\pi\alpha-\delta+\pi-\delta)(\alpha^{-1}t-C_\delta) \leq A(t).\]
Combining these upper and lower bounds, one sees that for $t< \min\{t_\varepsilon,t_\delta
\}$, 
\[
\left[\pi \left(\alpha^{-1}-1\right)-2\delta \alpha^{-1}-\varepsilon \right](-t) \leq \pi(1+\alpha)C_\delta+2\delta C_\delta\]
By choosing $\varepsilon,\delta
$ small, one can make $\left[\pi \left(\alpha^{-1}-1\right)-2\delta\alpha^{-1}-\varepsilon \right]>0$, and allowing $t\to-\infty$ yields a contradiction unless $\alpha=1$. Thus we conclude that $\alpha=1$.
\end{proof}

\subsection{Reflection symmetry}
We now exploit the maximality of the width of the limiting Grim hyperplane to deduce the following reflection symmetry.
\begin{theorem}\label{reflection}
Let $\{\Sigma_t\}_{t\leq 0}$ be a convex, ancient $O(n)$-symmetric solution to $\phi$-flow that lies in the slab $\{|x_1|\leq \pi/2\}$ and in no smaller slab. Then it is necessarily reflection symmetric about the hyperplane $\{x_1=0\}$. 
\end{theorem}
\begin{proof}
The argument is a standard application of the Alexandrov reflection principle, which holds for parabolic flows, and hence the same as \cite[Theorem 6.2]{BLT}. 
\end{proof}

\section{Area and displacement estimates} \label{sec:estimates}
We continue to study arbitrary convex $O(n)$-invariant ancient solutions lying in a slab. The aim here is to provide improved estimates for the enclosed area $A(t)=\int_{-t}^{0}\int_{\Gamma_\tau}\phi\,ds\,d\tau$ and the vertical displacement $\ell(t)$ for such solutions.

Since $\phi$ is non-degenerate, we may write
\begin{align*}              
    \phi(\kappa,\lambda)&=\kappa\phi(1,\lambda/\kappa)\\     &=\kappa\left(\phi(1,0)+\phi_\lambda(1,0)(\lambda/\kappa)+\phi_{\lambda\lambda}(1,\xi)(\lambda/\kappa)^2\right)\\
    &=\kappa\phi(1,0)+\lambda \dot{\phi}_1+\phi_{\lambda\lambda}(1,\xi)\lambda^2/\kappa\\
    &\leq \kappa+\lambda \dot{\phi}_1+C\lambda^2/\kappa\,,
\end{align*}
and hence
\begin{equation}\label{eq:phiapprox}  
    \kappa+\lambda \dot{\phi}_1-C\lambda^2/\kappa \leq \phi(\kappa,\lambda) \leq \kappa+\lambda \dot{\phi}_1+C\lambda^2/\kappa\
\end{equation}
where $0\leq\xi\leq\lambda/\kappa$ is given by the mean value theorem, $\dot{\phi}_1 \coloneqq \phi_\lambda(0,1)$ and $C\coloneqq \sup_{\xi\in[0,1]}\vert\phi_{\lambda\lambda}(1,\xi)\vert$. Thus the area estimate will reduce to estimates of $\int_{\Gamma_t}\lambda ds$ and $\int_{\Gamma_t}(\lambda^2/\kappa) ds$.

Due to the reflection symmetry of such solutions (Theorem \ref{reflection}), we have that the horizontal displacement
\[h(t) \coloneqq \max_{\theta \in S^1} \langle \gamma(\theta,t),e_1 \rangle\]
satisfies
\[h(t)=\langle\gamma(\pi/2,t),e_1\rangle=-\langle\gamma(-\pi/2,t),e_1\rangle\,.\]

Moreover, due to Lemma \ref{lambdaleqkappa}, we have
\[\kappa(\theta,t)\geq\lambda(\theta,t)\geq\lambda(\pi/2,t)=\kappa(\pi/2,t)\,,\]
which then implies that the minimum value of $\phi$ occurs at the poles, i.e.
\begin{equation}
    \min_{\theta \in S^1} \phi(\theta,t) =\phi(\pm \pi/2,t)\,.
\end{equation}

We obtain the desired estimates by using a graphical representation for the solution. The part of the curve $\Gamma_t$ with $y\geq0$ can be written as a graph $(x,u(x,t))$ with $x\in[-h(t),h(t)]$. In this representation,
\begin{equation}
    \frac{1}{4}\int_{\Gamma_t}\lambda ds=\int_0^{h(t)}\frac{1}{u(x,t)}dx
\end{equation}
and
\begin{equation}\label{lambda^2}
    \frac{1}{4}\int_{\Gamma_t}\frac{\lambda^2}{\kappa} ds=\int_0^{h(t)}\frac{\lambda^2}{\kappa}\frac{dx}{|{\cos\theta}|}=\int_0^{h(t)}\frac{|{\cos\theta}|}{\kappa u(x,t)^2}dx
\end{equation}
coming from the fact that $ds=\frac{dx}{\vert{\cos\theta}\vert}$ and $\lambda= \frac{\vert{\cos\theta}\vert}{u}$.

Moreover, the function $u$ satisfies the graphical $\phi$-flow equation
\[\frac{du}{dt}=-\phi\sqrt{1+u_x^2}.\]
Since $\phi$ is nondecreasing due to the Harnack inequality, and since the solution converges to the Grim reaper, we have the basic estimate
\[-\frac{du}{dt}\geq\phi(1,0)=1.\]
Now, given $x\in[0,\frac{\pi}{2})$, we may denote by $T(x)$ the time when the profile curve passes through the point $(x,0)$, i.e. $h(T(x))=x$, and hence $u(x,T(x))=0$. Now, integrating the previous inequality from $T$ to $t$, we get for all $x\in[0,h(t))$ that
\[u(x,t)\geq -t+T(x)\,.\]

Observe that
\[\kappa(\pi/2,t) \leq 2\frac{h(t)-x}{u(x,t)^2}\,,\]
which when combined with $\phi(\kappa,\lambda)\leq \phi(\kappa,\kappa)=\phi(1,1)\kappa=\phi_1\kappa$ yields
\begin{equation}
    \phi_{\min}(t) \leq 2\phi_1 \frac{h(t)-x}{u(x,t)^2}\,.
\end{equation}
Putting $x=0$, we get
\begin{equation}
    -\frac{dh}{dt}=\phi_{\min} \leq \frac{2\phi_1 h}{(-t)^2}\,.
\end{equation}
which, upon integration, gives
\begin{equation}
    h(t) \geq \frac{\pi}{2}e^{\frac{2\phi_1}{t}} \geq \frac{\pi}{2}\left(1-\frac{2\phi_1}{-t}\right).
\end{equation}
Setting $t=T(x)$ now gives
\begin{equation}
    u(x,t)\geq -t+T(x) \geq -t-\frac{\phi_1\pi}{\frac{\pi}{2}-x}\,.
\end{equation}
Thus we have the following estimate:
\begin{lemma}\label{lem:uhphiasymp}
For each $k\in\mathbb{N}$ there exists a constant $c_k$ such that for all $t<0$ and all $x\in[0,h(t)]$,
\begin{enumerate}[(i)]
    \item $\phi_{\min}(t) \leq \frac{\phi_1 \pi c_k}{(-t)^{k+1}}$ 
    \item $h(t) \leq \frac{\pi}{2}\left(1-\frac{2\phi_1 c_k}{(-t)^{k}}\right)$
    \item $u(x,t) \geq-t- \left(\frac{\phi_1 \pi c_k}{\frac{\pi}{2}-x}\right)^{1/k}$
\end{enumerate}
\end{lemma}
\begin{proof}
In view of the above estimates, we we may proceed by induction, as in \cite[Lemma 7.1]{BLT}.
\end{proof}

This allows us to estimate the integral that we wish to estimate:

\begin{lemma}\label{lem:lambdaintegrals}
    For any $\varepsilon>0$,
    \begin{enumerate}[(i)]
        \item
        $\int_{\Gamma_t}\lambda ds \leq \frac{2\pi}{-t}+o\left(\frac{1}{(-t)^{2-\varepsilon}}\right)$
        \item $\int_{\Gamma_t}\frac{\lambda^2}{\kappa} ds \leq O\left(\frac{1}{(-t)^{2}}\right)$
    \end{enumerate}
\end{lemma}
\begin{proof}
The first estimate may be proved as in \cite[Claim 7.2.1]{BLT}. The second estimate is not required in \cite{BLT} because the mean curvature $H=\kappa+(n-1)\lambda$ is linear in $\kappa$ and $\lambda$. Nonetheless, a similar idea works. 
Indeed, following \cite[Claim 7.2.1]{BLT}, we define $c(x,t)=\sqrt{\rho(t)^2-(x-(h(t)-\rho(t)))^2},$ where $\rho(t)=\frac{(-t)^2}{\pi}$, and set $\underline{x}(t)=\inf\{x\in[\pi/4,h(t)]:u(x,t)=c(x,t)\}$, and split the integral to be estimated as
\[\frac{1}{4}\int_{\Gamma_t}\frac{\lambda^2}{\kappa} ds=\int_0^{h(t)}\frac{\lambda^2}{\kappa}\frac{dx}{\vert{\cos\theta}\vert}=\int_0^{\underline x(t)}\frac{\lambda^2}{\kappa}\frac{dx}{\vert{\cos\theta}\vert}+\int_{\underline x(t)}^{h(t)}\frac{\lambda^2}{\kappa}\frac{dx}{\vert{\cos\theta}\vert}.\]
Since $\lambda/\kappa \leq 1$,
\begin{equation}\label{lamb^2pt1}
    \int_{\underline x(t)}^{h(t)}\frac{\lambda^2}{\kappa}\frac{dx}{\vert{\cos\theta}\vert}\leq \int_{\underline x(t)}^{h(t)}\frac{\lambda\,dx}{\vert{\cos\theta}\vert} =o(t^{-2})
\end{equation}
just as in \cite[Claim 7.2.1]{BLT}. The remaining term is estimated as follows. Let $k\in\mathbb{N}$ and choose $x_0(t)=\frac{\pi}{2}-\frac{\pi nc_k}{(-t)^k}$. If $\underline x(t)\leq x_0(t)$, then
\begin{equation*}
     \int_0^{\underline x(t)}\frac{\lambda^2}{\kappa}\frac{dx}{\vert{\cos\theta}\vert} \leq \int_0^{x_0(t)}\frac{\lambda^2}{\kappa}\frac{dx}{\vert{\cos\theta}\vert}\leq \int_0^{x_0(t)}\frac{1}{u(x,t)^2}dx
\end{equation*}
and if not,
\begin{align*}
    \int_0^{\underline x(t)}\frac{\lambda^2}{\kappa}\frac{dx}{\vert{\cos\theta}\vert} &\leq \int_0^{x_0(t)}\frac{\lambda^2}{\kappa}\frac{dx}{\vert{\cos\theta}\vert}+\int_{x_0(t)}^{\underline x(t)}\frac{\lambda^2}{\kappa}\frac{dx}{\vert{\cos\theta}\vert}\\
    &\leq \int_0^{x_0(t)}\frac{\lambda^2}{\kappa}\frac{dx}{\vert{\cos\theta}\vert}+\int_{x_0(t)}^{\underline x(t)}\lambda\frac{dx}{\vert{\cos\theta}\vert}\\
    &\leq \int_0^{x_0(t)}\frac{1}{u(x,t)^2}dx +\int_{x_0(t)}^{h(t)}\frac{dx}{c(x,t)}\,.
\end{align*}
Either way, Lemma \ref{lem:uhphiasymp} (cf. \cite[Claim 7.2.1]{BLT}) yields
\begin{align}\label{lamb^2pt2}
    \int_0^{\underline x(t)}\frac{\lambda^2}{\kappa}\frac{dx}{\vert{\cos\theta}\vert} 
    &\leq \frac{x_0(t)}{u(x_0(t),t)^2}+\int_{x_0(t)}^{h(t)}\frac{dx}{c(x,t)}\nonumber\\
    &\leq \frac{\pi}{2}\left(\frac{1}{-t}+o((-t)^{-1})\right)^2+o((-t)^{-2})\nonumber\\
    &\leq O((-t)^{-2})\,.
\end{align}
Putting equations (\ref{lamb^2pt1}) and (\ref{lamb^2pt2}) together yields
\[\frac{1}{4}\int_{\Gamma_t}\frac{\lambda^2}{\kappa} ds \leq O((-t)^{-2})
\]
as claimed.
\end{proof}

This leads us to the following area estimate:
\begin{corollary}
    \begin{equation}
        -t+\dot{\phi}_{1}\log(-t)-C\leq \frac{A(t)}{2\pi} \leq -t+\dot{\phi}_{1}\log(-t) +C
    \end{equation} 
\end{corollary}

\begin{proof}
The upper bound follows from integrating $A'$ in (\ref{eq:evolve A}), estimating $\phi$ using the upper bound in equation (\ref{eq:phiapprox}) and using Lemma \ref{lem:lambdaintegrals}. The lower bound will be inferred from the upper bound via an application of H\"{o}lder's inequality. Indeed,
\[h^2(t)\leq \int_0^{h(t)}u(x,t)dx \int_0^{h(t)}u(x,t)^{-1} dx=\frac{A(t)}{16}\int_{\Gamma_t}\lambda ds\,,\]
so that Lemma \ref{lem:uhphiasymp} gives
\[\int_{\Gamma_t}\lambda ds \geq \frac{2\pi\left(1-\frac{2\phi_1c_1}{-t}\right)^2}{-t+\dot{\phi}_{1}\log(-t) +C} \geq \frac{2\pi}{-t}-o\left(\frac{1}{(-t)^{2-\epsilon}}\right).\]
Thus, using the lower bound in \eqref{eq:phiapprox} gives us the desired lower bound for area.
\end{proof}   

Now the following refinement to the displacement estimate is obtained:
\begin{lemma}\label{lem:ooft^-eps}
    For any $\varepsilon\in (0,1)$,
    \begin{equation}
        \ell(t) \leq -t + o((-t)^{\varepsilon}).
    \end{equation}
\end{lemma}
\begin{proof}
    This follows from the same area estimation trick of \cite[Lemma 7.3]{BLT}.
\end{proof}


\begin{corollary}\label{cor:improvedHestimate} For any $\varepsilon\in(0,1)$
\[
\phi(\theta, t)\ge \vert{\cos\theta}\vert\left(1+\frac{\dot{\phi}_1}{-t}-o\left(\frac{1}{(-t)^{2-\varepsilon}}\right)\right)\;\;\text{as}\;\; t\to-\infty\;\;\text{for all}\;\; \theta \in S^1\,,
\]
where $\dot\phi_1$ is defined by \eqref{eq:phiapprox}.
\end{corollary}
\begin{proof} By symmetry, it suffices to prove the claim for $\theta\in (-\pi/2,\pi/2)$ (note that for $\theta=\pm\pi/2$ the claim holds trivially). Consider the function $w:(-\frac{\pi}{2},\frac{\pi}{2})\times(-t,0)\to\R$ defined by $w(\theta,t):=f(t)\cos\theta$, where the function $f:(-\infty,0)\to\R_+$ will be determined momentarily. Observe that
\[
w_t=\phi_\kappa\kappa^2w_{\theta\theta}+|\mathrm{II}|_\phi^2w+\left(\frac{f'}{f}-\phi_\lambda\lambda^2\right)w\,,
\]
where $\vert \mathrm{II}\vert_\phi^2:=\phi_\kappa\kappa^2+\phi_\lambda\lambda^2$.
Recalling Lemma \eqref{evoeq}, we compute
\begin{align*}
\frac{w}{\phi}\left(\partial_t-\phi_\kappa\kappa^2\partial^2 _{\theta}\right)\frac{\phi}{w}={}&2\phi_\kappa\kappa^2\frac{w}{\phi}\left(\frac{\phi}{w}\right)_\theta\frac{w_\theta}{w}-\phi_\lambda\lambda^2\tan\theta \frac{\phi_\theta}{\phi}-\left(\frac{f'}{f}-\phi_\lambda\lambda^2\right)\,.
\end{align*}
Rewriting
\[
\frac{\phi_\theta}{\phi}=\frac{w}{\phi}\left(\frac{\phi}{w}\right)_\theta-\frac{w_\theta}{w}\quad\text{and}\quad
\frac{w_\theta}{w}=-\tan\theta
\]
we obtain
\begin{align}\label{eq:evolveH/w}
\frac{w}{\phi}\left(\partial_t-\phi_\kappa\kappa^2\partial^2 _{\theta}\right)\frac{\phi}{w}+{}&\tan\theta\frac{w}{\phi}\left(\frac{\phi}{w}\right)_\theta\left(\phi_\lambda\lambda^2+2\phi_\kappa\kappa^2\right)\nonumber\\
={}&\phi_\lambda\sec^2\theta\lambda^2-\frac{f'}{f}\nonumber\\
={}&\frac{\phi_\lambda}{y^2}-\frac{f'}{f}\nonumber\\
\ge{}&\frac{\dot{\phi}_1}{y^2}-C\frac{\lambda}{\kappa y^2}-\frac{f'}{f}\nonumber\\
\ge{}&\frac{\phi_\lambda}{\ell^2}-\frac{C}{\ell^3}-\frac{f'}{f}
\end{align}
since $\kappa\ge C\phi\ge \cos\theta$.

Now fix any $\varepsilon\in(0,1)$. By Lemma \ref{lem:ooft^-eps}, there is some $C_\varepsilon<\infty$ such that
\[
\frac{1}{\ell(t)^2}\geq\frac{1}{(-t)^2}-\frac{C_\varepsilon}{(-t)^{3-\varepsilon}}
\]
for all $t\in(-\infty,-1]$, say. Thus, if we set
\[
f(t):=\exp\left(\left[\frac{\dot\phi_1}{-t}-\frac{\Lambda}{(2-\varepsilon)(-t)^{2-\varepsilon}}\right]\right)\,,
\]
then we may choose $\Lambda$ so large that
\[
\frac{f'(t)}{f(t)}=\left[\frac{\dot\phi_1}{(-t)^2}-\frac{\Lambda}{(-t)^{3-\varepsilon}}\right]\leq \frac{\dot\phi_1}{\ell(t)^2}-\frac{C}{\ell^3}
\]
for all $t<-1$. So the maximum principle yields
\[
\min_{S^1\times\{t\}}\frac{\phi}{w}\geq\min_{S^1\times\{t_0\}}\frac{\phi}{w}\quad\text{for all}\quad -1>t>t_0\,.
\]
But the right hand side approaches 1 as $t\to-\infty$. The claim follows by estimating $\exp(\zeta)\geq 1+\zeta$. 
\end{proof}

Integrating the lower speed bound yields a displacement estimate.
\begin{lemma}\label{lem:ellasymptotics}
The limit
\[
C:=\lim_{t\to-\infty}(\ell(t)+t-\dot{\phi}_1\log(-t))
\]
exists (in the extended real line $\R\cup\{\infty\}$).
\end{lemma}
\begin{proof}
Given any $\varepsilon\in(0,1)$ set $f(t):=\frac{C}{1-\varepsilon}\frac{1}{(-t)^{1-\varepsilon}}$ for some $C\in\R$. By Corollary \ref{cor:improvedHestimate}, we can choose $C$ so that
\[
\frac{d}{dt}(\ell+t-\dot{\phi}_1\log(-t)-f)\leq 0\,.
\]
The claim follows because $\lim_{t\to-\infty}f=0$.
\end{proof}

\section{Uniqueness} \label{sec:uniqueness}
In this section, we will show that the pancake solutions constructed in Section 4 are unique. We will do so in several steps. The essential idea, as in \cite{BLT}, is to exploit the constructed solution as a barrier using the Alexandrov reflection principle. First, we show that on the constructed solution the constant $C$ defined by
\[
C\coloneqq \lim_{t \to -\infty}\left(\ell(t)+t-\dot{\phi}_1\log(-t)\right)
\]
is finite. We will use methods developed in Section \ref{sec:estimates} to obtain area and displacement estimates, which will allow us to prove this. We require the following lemma.

\begin{lemma}\label{lem:phi ge cos theta}
On the approximating solutions, $\phi_R(\theta,t)\geq |{\cos\theta}|$.
\end{lemma}
\begin{proof}
First observe that $\phi_R(\theta,-T_R)\geq |{\cos\theta}|$ by construction, and the inequality is always (trivially) satisfied at the poles $\theta=-\pi/2,\pi/2$. So we need to show that it continues to hold in $(-\pi/2,\pi/2)$ (by reflection symmetry, it will hold on the other half as well). To that end, consider the function $w=\phi/\cos\theta$. Suppose that $w$ develops a new interior minimum at some $(\theta_0,t_0)$. At this point, we have
\[
\phi_\theta=-\phi \tan\theta_0
\]
and hence
\[
0 \geq w_t-\kappa^2\phi_k w_{\theta\theta}=\phi_\lambda\lambda^2\phi\sec^3\theta_0 > 0.
\]
This leads to a contradiction and the claim follows. 
\end{proof}

Equations \eqref{DispEst} now allow us to proceed as in \S \ref{sec:estimates} to obtain the area estimate 
\begin{equation}
    -t+ \Dot{\phi}_1\log(-t)+C \geq \frac{A(t)}{2\pi} \geq -t+ \Dot{\phi}_1\log(-t)-C
\end{equation}
for all $t\in[-T_R,1)$. Bounding the area of $\Gamma_R(t)$ by a rectangle of height $2\ell_R(t)$ and width $\pi$ gives the length estimate
\[\ell_R(t) \geq -t+ \Dot{\phi}_1\log(-t)-C.\]

Now we show that $\ell_R(t)$ is bounded similarly from above (uniformly with respect to $R$). Proceeding in the same way as Lemma \ref{lem:ooft^-eps}, we obtain that for any $\varepsilon\in (0,1)$,
\[\ell(t) \leq -t+C_\varepsilon(-t)^{\varepsilon}\,,\]
where $C_\varepsilon$ is independent of $R.$
Now, proceeding as in \cite{BLT}, we may choose $C<\infty$ such that
\[\frac{d}{dt}(\ell_R(t)+t-\Dot{\phi}_1\log(-t)-f) \leq 0,\]
where $f \coloneqq \frac{C}{1-\varepsilon}\frac{1}{(-t)^{1-\varepsilon}}.$ Now integration yields
\[\ell_R(t)+t-\Dot{\phi}_1\log(-t) \leq \ell_R(-T_R)-T_R-\Dot{\phi}_1\log(T_R)+f(t)-f(-T_R).\]
Now, since
\[\ell_R(-T_R) \leq T_R + \Dot{\phi}_1\log(-T_R)+C+\log 2\]
and $\lim_{t \to -\infty}f(t)=0$, we let $R \to \infty$ to see that $C$ is finite.

Now we sketch the proof of uniqueness (cf.  \cite[Proof of Theorem 1.2]{BLT} for details.)  Suppose $\gamma,\gamma':S^1\times(-\infty,0)\to\R^2$ are the turning angle parametrizations of the profile curves $\Gamma_t,\Gamma_t'$ respectively of the solution constructed in Section 3, and an arbitrary $O(n)$-invariant compact convex ancient solution to the $\phi$-flow that lies in the same slab and no smaller slab. Let $\Omega_t,\Omega'_t$ be the open sets contained inside the profile curves. Since both solutions contract to the origin at $t=0$, they must intersect for all previous times due to the avoidance principle. Now consider the two constants
\begin{align*}
    C&:=\lim_{t\to-\infty}(\ell(t)+t-\dot{\phi}_1\log(-t))\\
    C'&:=\lim_{t\to-\infty}(\ell'(t)+t-\dot{\phi}_1\log(-t))
\end{align*}
where $\ell(t),\ell'(t)$ are the respective vertical displacements of $\gamma,\gamma'$. We have just shown that $C$ is finite, and that $C'\in \R\cup\{\infty\}$. If $C \neq C'$, it follows by the strong maximum principle that $\Omega_t'\Subset \Omega_t$ or $\Omega_t\Subset \Omega_t'$, and hence $\Gamma_t,\Gamma_t'$ never intersect. Thus $C=C'$. This implies that by the same reasoning, for any $\tau>0$ we would have that $\Omega_t'\Subset\Omega_{t+\tau}$. Taking $\tau \to 0$ shows that $\bar\Omega_t' \subset \bar\Omega_t$. By symmetry,  $\bar\Omega_t \subset \bar\Omega_t'$. Thus the solutions coincide.

\bibliographystyle{plain}
\bibliography{references.bib}

\end{document}